\documentclass[graybox]{svmult}

\usepackage{type1cm}        
\usepackage{makeidx}         
\usepackage{graphicx}        
\usepackage{multicol}        
\usepackage[bottom]{footmisc}

\usepackage{newtxtext}       %
\usepackage[varvw]{newtxmath}       

\makeindex             

\usepackage{graphicx} 
\usepackage[english]{babel}
\usepackage{amsmath,amsthm}
\usepackage{amsfonts}
\usepackage{comment}
\usepackage{enumerate}
\usepackage{graphicx}
\usepackage{hyperref}
\usepackage{subcaption}
\usepackage{color}
\usepackage[all]{xy}
\usepackage{thmtools, thm-restate}
\usepackage{overpic}

\usepackage{amssymb,tikz}
\usepackage{quiver}
\usepackage{enumitem}
\usepackage{multicol}
\usepackage{alltt}
\usetikzlibrary{positioning}
\usetikzlibrary{knots,decorations.pathreplacing}
\usepackage{ifthen,fp,calc}
\usetikzlibrary{patterns}

\usepackage[top=1.23in, bottom=1.22in, left=1.23in, right=1.23in]{geometry}

\newtheorem{thm}{Theorem}[section]
\newtheorem{cor}[thm]{Corollary}
\newtheorem{lem}[thm]{Lemma}

\newtheorem{ques}[thm]{Question}

\theoremstyle{definition}
\newtheorem{defn}[thm]{Definition}

\theoremstyle{remark}

\numberwithin{equation}{section}

\newcommand{\spin}{\ifmmode{\rm Spin}\else{${\rm spin}$\ }\fi}
\newcommand{\spinc}{\ifmmode{{\rm Spin}^c}\else{${\rm spin}^c$}\fi}

\newcommand{\Z}{\mathbb{Z}}

\DeclareMathOperator{\tors}{tors}

\newcommand{\plumb}{\entrymodifiers={+[o][F-]} \xymatrix@C=8pt}
\newcommand{\el}{\ar@{-}[r]}
\newcommand{\ed}{\ar@{..}[r] }

\begin{document}

\title*{A survey on embeddings of 3-manifolds in definite 4-manifolds}
\author{Paolo Aceto, Duncan McCoy, and JungHwan Park}
\institute{Paolo Aceto \at Université de Lille, \email{paoloaceto@gmail.com}
\and Duncan McCoy \at Universit\'{e} du Qu\'{e}bec \`{a} Montr\'{e}al \email{mc\_coy.duncan@uqam.ca}
\and JungHwan Park \at Korea Advanced Institute of Science and Technology \email{jungpark0817@kaist.ac.kr}
}
%
%
\maketitle

\abstract*{}

\abstract{
This article presents a survey on the topic of embedding 3-manifolds in definite 4-manifolds, emphasizing the latest progress in the field. We will focus on the significant role played by Donaldson's diagonalization theorem and the combinatorics of integral lattices in understanding these embeddings. Additionally, the article introduces a new result concerning the embedding of amphichiral lens spaces in negative-definite manifolds.}

\section{Introduction}
Every closed 3-manifold smoothly embeds in $S^5$~\cite{Hirsch:1961-1, Rohlin:1965-1, Wall:1965-1}, but the question of codimension 1 embedding is much more delicate. In this survey article, we present recent progress on known results about the embedding problem. We divide it into two parts: cases where the 4-manifold is $S^4$ and cases involving definite manifolds. Although there are results on the embedding of 3-manifolds into indefinite manifolds (see, e.g., \cite{Kawauchi:1988-1, Aceto-Golla-Larson:2017-1}), we focus exclusively on the definite case.

Most of the results included in this article demonstrate that embedding is not possible in certain cases. To obtain such results, many cutting-edge technologies have been applied over the years, such as Casson-Gordon signatures~\cite{Gilmer-Livingston:1983-1}, the G-index theorem~\cite{Crisp-Hillman:1998-1}, and various gauge-theoretic techniques (see~\cite[Section 2]{Budney-Burton:2008-1} for a comprehensive survey). Here, we focus on the significant role played by Donaldson's diagonalization theorem~\cite{Donaldson:1987-1} and the combinatorics of integral lattices in understanding these embeddings. Even though the article focuses on the smooth embedding of 3-manifolds in 4-manifolds, we point out that there are many interesting results about topological embeddings as well (see e.g.\ \cite{Edmonds:2005-1, Edmonds-Livingston:1996-1}). Additionally, the article introduces a new result that provides constructions for embedding amphichiral lens spaces in negative-definite manifolds, which we will now state.

Given coprime integers $p$ and $q$ with $p > q > 0$, let $L(p,q)$ be an \emph{amphichiral} lens space; that is, $L(p,q)$ and $-L(p,q)$, the manifold $L(p,q)$ with reversed orientation, are orientation-preservingly diffeomorphic. It turns out that amphichiral lens spaces always have small negative-definite fillings, implying that they can be embedded in small, closed, smooth, negative-definite 4-manifolds.


\begin{thm}\label{thm:amhichiralfilling}
Every amphichiral lens space admits a smooth, negative-definite filling $W$ with $b_1(W) = 0$ and $b_2(W) = 2$. In particular, every amphichiral lens space smoothly embeds in a closed, smooth, negative-definite manifold $X$ such that $b_1(X) = 0$ and $b_2(X) = 4$.\end{thm}

This theorem should be compared to \cite[Theorem 1.5]{Aceto-McCoy-Park:2022-1}, which implies that for any positive integer $k$, there are infinitely many lens spaces that do not allow smooth embedding into any closed, negative-definite manifold $X$ with $b_2(X) \leq k$. Moreover, the theorem also provides an upper bound on the smooth 4-genus for amphichiral 2-bridge knots as a corollary:

\begin{cor}\label{cor:amphichiralgenus}
The smooth 4-dimensional genus of an amphichiral 2-bridge knot is at most one.
\end{cor}	

\noindent Note that this corollary provides a complete computation of the 4-genus for amphichiral 2-bridge knots, as Lisca classified smoothly slice 2-bridge knots~\cite[Theorem 1.2]{Lisca:2007-1}.



\subsection*{Notation and conventions}
In this paper, every 3-manifold is smooth, connected, closed, and oriented. All 4-manifolds are smooth, connected, compact, and oriented; furthermore, all maps between manifolds are smooth. We denote by $-Y$ the manifold $Y$ with its orientation reversed. For two manifolds $Y_1$ and $Y_2$, the symbol $Y_1 \cong Y_2$  is used to indicate that there is an orientation-preserving diffeomorphism between them.

\subsection*{Acknowledgements}
JP is partially supported by Samsung Science and Technology Foundation (SSTF-BA2102-02) and the POSCO TJ Park Science Fellowship.

\section{Preliminaries}

\subsection{Lens spaces}
We define the \emph{lens space} $L(p,q)$ to be the result of $-p/q$-Dehn surgery on the unknot in $S^3$ where $p$ and $q$ are coprime integers with $p > q > 0$. Note that $L(p,q)$ is a rational homology sphere with $H_1(L(p,q);\Z) \cong \Z/p\Z$. Each lens space $L(p,q)$ serves as the boundary of the \emph{canonical negative-definite plumbing} $X(p,q)$. This plumbing is a smooth, negative-definite 4-manifold obtained by plumbing disk-bundles over spheres according to a \emph{linear plumbing graph} $\Gamma$ of the form:

\[
\begin{tikzpicture}[xscale=1.0, yscale=1, baseline={(0,0)}]
    \foreach \x/\label in {1/-a_1,2/-a_2,3/-a_3,5/-a_{n}} {
        \node at (\x-0.1, .4) {$\label$};
    }
    
    \node at (1, 0) (A1) {$\bullet$};
    \node at (2, 0) (A2) {$\bullet$};
    \node at (3, 0) (A3) {$\bullet$};
    \node at (4, 0) (A4) {$\cdots$};
    \node at (5, 0) (An) {$\bullet$};
    
    \draw (A1) -- (A2) -- (A3) -- (A4) -- (An);
\end{tikzpicture}
\]
where the weights $a_i$ satisfy $a_i\geq 2$ and are uniquely determined by the continued fraction expansion
$$\frac{p}{q} = a_1 - \cfrac{1}{a_2 - \cfrac{1}{\ddots - \cfrac{1}{a_n}}}$$
of $p/q$. In this case, we write $p/q = [a_1, \dots, a_n]^-$ to indicate that we are working with negative continued fractions.

Additionally, a connected sum of lens spaces $\#_i L(p_i,q_i)$, is the boundary of the boundary connected sum of canonical negative-definite plumbings $\natural_i X(p_i,q_i)$. The plumbing graph associated with this construction is the disjoint union of the linear plumbing graphs for each of the summands.

For each lens space $L(p,q)$, there exists a unique 2-bridge link, denoted $K(p,q)$, such that its double-branched cover, $\Sigma_2(K(p,q))$, is diffeomorphic to $L(p,q)$~\cite{Hodgson-Rubinstein:1985-1}. It should be noted that a 2-bridge link $K(p,q)$ is either a knot or it has two components. Furthermore, $K(p,q)$ is a knot if and only if $p$ is odd.

\subsection{Seifert fibered spaces}
Let $e$ be an integer and $(p_1,q_1),(p_2,q_2),\ldots, (p_n,q_n)$ be pairs of coprime integers. The \emph{Seifert fibered space} over $S^2$, $$Y= S^2(e,(p_1,q_1),(p_2,q_2),\ldots, (p_n,q_n)),$$ has the following surgery presentation:

$$\begin{tikzpicture}
\begin{knot}[clip width=4, clip radius=2pt,ignore endpoint intersections=false,]
\draw[color=black] (0.5,-1) node {$\frac{p_1}{q_1}$}; 
\draw[color=black] (1.5,-1) node {$\frac{p_2}{q_2}$}; 
\draw[color=black] (2.5,-.6) node {$\cdots$}; 
\draw[color=black] (3.5,-1) node {$\frac{p_n}{q_n}$}; 
\draw[color=black] (2,2.4) node {$e$}; 
\strand[thick] (0,0)--(4,0) to [out=right,in=down] (5,1)to [out=up,in=right] (4,2)--(0,2) to [out=left,in=up] (-1,1) to [out=down,in=left] (0,0);
\strand[thick] (0.5,0) ellipse (0.2cm and 0.5cm);
\strand[thick] (1.5,0) ellipse (0.2cm and 0.5cm);
\strand[thick] (3.5,0) ellipse (0.2cm and 0.5cm);
\flipcrossings{2,4,6,8}
\end{knot}\end{tikzpicture}$$
The \emph{generalized Euler invariant} of $Y$ is defined as $$\varepsilon(Y) := e - \sum_{i=1}^{n} \frac{q_i}{p_i}.$$

\noindent After possibly changing the orientation of $Y$, we may assume that $Y$ is in \emph{standard form}, that is, $\varepsilon(Y) \leq 0$ and $p_i/q_i < -1$. With these conventions, the Poincar\'{e} homology sphere that bounds a negative-definite manifold is $S^2(-2,(2,-1),(3,-2), (5,-4))$. If $Y$ is in standard form, then it bounds a plumbed negative semidefinite 4-manifold prescribed by the following plumbing graph $\Gamma$:


\[
\begin{tikzpicture}[xscale=1.2, yscale=1.2, baseline={(0,0)}]

    \foreach \x/\label in {1/ a^1_1,2/a^1_2,3/a^1_3,5/a^1_{h_1}} {
        \node at (\x, 1.25) {$\label$};
    }

\foreach \x/\label in {1/ a^2_1,2/a^2_2,3/a^2_3,5/a^2_{h_2}} {
        \node at (\x, .4) {$\label$};}

\foreach \x/\label in {1/ a^n_1,2/a^n_2,3/a^n_3,5/a^n_{h_n}} {
        \node at (\x, -2.1) {$\label$};}

    \node at (-.4, 0) {$e$};
    
    \node at (0, 0) (A0) {$\bullet$};
    \node at (1, 0) (A1) {$\bullet$};
    \node at (2, 0) (A2) {$\bullet$};
    \node at (3, 0) (A3) {$\bullet$};
    \node at (4, 0) (A4) {$\cdots$};
    \node at (5, 0) (An) {$\bullet$};

    \node at (1, .85) (B1) {$\bullet$};
    \node at (2, .85) (B2) {$\bullet$};
    \node at (3, .85) (B3) {$\bullet$};
    \node at (4, .85) (B4) {$\cdots$};
    \node at (5, .85) (Bn) {$\bullet$};
        
    \node at (1, -1.7) (D1) {$\bullet$};
    \node at (2, -1.7) (D2) {$\bullet$};
    \node at (3, -1.7) (D3) {$\bullet$};
    \node at (4, -1.7) (D4) {$\cdots$};
    \node at (5, -1.7) (Dn) {$\bullet$};

    \node at (3, -0.75) {$\vdots$};
    
    \draw (A0) -- (B1);
    \draw (A0) -- (D1);
    \draw (A0) -- (A1) -- (A2) -- (A3) -- (A4) -- (An);
     \draw  (B1) -- (B2) -- (B3) -- (B4) -- (Bn);
     \draw  (D1) -- (D2) -- (D3) -- (D4) -- (Dn);
\end{tikzpicture}
\]
where we use the unique continued fraction expansion $p_i/q_i = [a^i_1, \dots, a^i_{h_i}]^-$  for each $i$, such that $a^i_j \leq -2$. The smooth negative semidefinite plumbed  4-manifold corresponding to $\Gamma$ is denoted by $X_\Gamma$.


Brieskorn manifolds introduced by Brieskorn~\cite{Brieskorn:1966-1, Brieskorn:1966-2} are Seifert fibered spaces. Let $p, q,$ and $r$ be pairwise relatively prime positive integers. Then the \emph{Brieskorn homology sphere} $\Sigma(p,q,r)$ is defined as the link of the following singularity:

$$\Sigma(p,q,r) := \{ x^p+y^q+z^r=0 \} \cap S^5 \subset \mathbb{C}^3.$$
The Poincar\'{e} homology sphere mentioned above is $\Sigma(2,3,5)$, and more generally each $\Sigma(p,q,r)$ is diffeomorphic to $S^2(e,(p,q_1),(q,q_2),(r,q_3))$ for some $q_1, q_2, q_3$, and $e$~\cite{Neumann-Raymond:seifert-manifolds}.

\subsection{Integral lattices and map of lattices}
An \emph{integral lattice} is a pair $(L, Q_L)$, where $L$ is a free abelian group, and $Q_L \colon L \times L \rightarrow \mathbb{Z}$ is a symmetric bilinear pairing. A \emph{map of lattices} is a group homomorphism
$$\phi: (L, Q_L) \rightarrow (L', Q_{L'})$$
that preserves the bilinear pairings. In particular, the intersection form of a compact oriented 4-manifold $X$ defines an integral lattice $(H_2(X)/ \tors, Q_X)$. We will often refer to this lattice simply as $Q_X$.

\begin{defn}\label{def:QGamma}
Let $\Gamma$ be a disjoint union of linear plumbing graphs, and let $V(\Gamma)$ be the set of its vertices. We define the \emph{lattice associated to $\Gamma$} as the lattice
$$(\mathbb{Z} V(\Gamma), Q_{\Gamma}),$$
where the group is the free abelian group generated by the vertices of $\Gamma$, and the pairing $Q_\Gamma$ is defined by the following rule: for all vertices $u, v \in V(\Gamma)$,
\[
Q_{\Gamma}(u,v)= \begin{cases}
w(v) & \text{if $u = v$},\\
-1 & \text{if $u, v$ are adjacent},\\
0 & \text{otherwise},
\end{cases}
\]
where $w(v)$ denotes the weight of $v$.
For brevity, we often refer to $(\mathbb{Z} V(\Gamma), Q_{\Gamma})$ simply as $Q_\Gamma$.
\end{defn}

Note that if $\Gamma$ is the linear plumbing graph of $X(p,q)$, or more generally of the boundary of the boundary connected sum of canonical negative-definite plumbings $\natural_i X(p_i,q_i)$, then $(\mathbb{Z} V(\Gamma), Q{\Gamma})$ is the integral lattice $\left(H_2(\natural_i X(p_i,q_i))/ \tors, Q_{\natural_i X(p_i,q_i)}\right)$, and the same is true for the Seifert fibered spaces.

\section{Embeddings in $S^4$}
\subsection{Embedding sums of lens spaces in $S^4$}
Hantzsche~\cite{Hantzsche:1938-1} used the Mayer-Vietoris sequence and duality to show that if a rational homology 3-sphere $Y$ embeds in $S^4$, then
$$H_1(Y;\Z)\cong G \oplus G,$$
where $G$ is a finite abelian group. In particular, we see that no lens space embeds in $S^4$.

Let $L(p,q)$ be a lens space with coprime integers $p$ and $q$ with $p>q>0$. If we consider a punctured lens space $L(p,q)_0$, the situation is quite different. Zeeman~\cite{Zeeman:1965-1} proved that for any knot $K$ in $S^3$, we can obtain a smooth 2-knot in $S^4$ by performing an operation called \emph{1-twist spinning} of $K$, denoted by $S_1(K)$. Roughly, this construction is achieved by spinning a knotted arc obtained by removing a small segment from the given $K$ along the $\mathbb{R}^2$-axis and performing a full rotation during the spinning process. Zeeman proved that $S_1(K)$ is smoothly unknotted for any choice of $K$. In particular, the knot $K \mathbin{\#} -K$ is a section of an smoothly unknotted 2-knot $S_1(K)$, and its double-branched cover $\Sigma_2(K \mathbin{\#} -K) \cong \Sigma_2(K) \mathbin{\#} -\Sigma(K)$ smoothly embeds in $S^4$. In particular, when $L(p,q)$ is the double-branched cover of a 2-bridge knot, that is, when $p$ is odd, we have that $L(p,q) \mathbin{\#} - L(p,q)$ smoothly embeds in $S^4$. It is important to note that the existence of an embedding of $Y \mathbin{\#} -Y$ is equivalent to the existence of an embedding of a punctured $Y_0$ in $S^4$. This equivalence is because the tubular neighborhood of $Y_0$ is $Y_0 \times I$, and its boundary is $Y \mathbin{\#} -Y$. Epstein~\cite{Epstein:1965-1} used a homotopy theoretical argument to prove that if $p$ is even, then $L(p,q)_0$ never smoothly embeds in $S^4$. In summary, we have the following:

\begin{thm}[{\cite{Zeeman:1965-1,Epstein:1965-1}}]\label{Thm:puncturedlensspaceinS4}
Let $L(p,q)$ be a lens space and $L(p,q)_0$ be a 3-manifold obtained by puncturing $L(p,q)$. Then the following are equivalent:
\begin{enumerate}[label=(\roman*), font=\upshape]
\item $p$ is odd;
\item $L(p,q) \mathbin{\#} - L(p,q)$ smoothly embeds in $S^4$;
\item $L(p,q)_0$ smoothly embeds in $S^4$.
\end{enumerate}
\end{thm}

Naturally, one is led to understanding embeddings of general connected sums of lens spaces in $S^4$. The first progress was obtained by Kawauchi and Kojima~\cite{Kawauchi-Kojima:1980-1} using the linking form of the lens spaces. They proved:
\begin{thm}[{\cite[Proposition 6.1]{Kawauchi-Kojima:1980-1}}] Let $L(p,q)$ and $L(p,q')$ be lens spaces. If $L(p,q) \mathbin{\#} -L(p,q')$ smoothly embeds in $S^4$, then $L(p,q)$ and $L(p,q')$ are homotopy equivalent.
\end{thm}
 
Using the Casson-Gordon invariant of finite cyclic covers of 3-manifolds~\cite{Casson-Gordon:1986-1, Casson-Gordon:1978-1, Gilmer:1981-1, Gordon:1978-1}, Gilmer and Livingston~\cite{Gilmer-Livingston:1983-1} made further progress as follows. Recall that two rational homology 3-spheres $Y_1$ and $Y_2$ are \emph{homology cobordant} if there exists a smooth compact oriented 4-manifold $W$ with boundary $\partial W = Y_1 \mathbin{\cup} -Y_2$, such that the inclusions $Y_i \hookrightarrow W$ induce isomorphisms $H_*(Y_i; \mathbb{Z}) \rightarrow H_*(W; \mathbb{Z})$ for $i=1,2$.

\begin{thm}[{\cite[Theorem 3.4]{Gilmer-Livingston:1983-1}}] Let $L(p,q)$ and $L(p,q')$ be lens spaces. If $L(p,q) \mathbin{\#} -L(p,q')$ smoothly embeds in $S^4$, then $L(p,q)$ and $L(p,q')$ are homology cobordant. Moreover, if $p$ is a prime power or less than $231$, then $L(p,q)$ and $L(p,q')$ are diffeomorphic.
\end{thm}

Fintushel and Stern \cite{Fintushel-Stern:1987-1} reinterpreted the Casson-Gordon invariants via Yang-Mills theory~\cite{Donaldson:1983-1, Donaldson:1986-1, Fintushel-Stern:1984-1, Fintushel-Stern:1985-1, Fintushel-Stern:1987-2}. As a consequence, they proved the following:

\begin{thm}[{\cite[Theorem 6.3]{Fintushel-Stern:1987-1}}] If $L(p,q)$ and $L(p,q')$ are lens spaces with $p$ odd, then they are homology cobordant if and only if they are diffeomorphic. In particular, if $L(p,q) \mathbin{\#} -L(p,q')$ smoothly embeds in $S^4$, then $L(p,q)$ and $L(p,q')$ are diffeomorphic.
\end{thm}

The most general statement was obtained quite recently by Donald~\cite{Donald:2015-1}. He provided a complete characterization of connected sums of lens spaces that can be embedded in $S^4$.

\begin{thm}[{\cite[Theorem 1.1]{Donald:2015-1}}]\label{thm:donald} Let $\{L(p_i,q_i)\}_{1\leq i \leq n}$ be a collection of lens spaces. If $L = \#_{i=1}^n L(p_i,q_i)$, then $L$ smoothly embeds in $S^4$ if and only if each $p_i$ is odd and there exists a closed 3-manifold $Y$ such that $L$ is diffeomorphic to $Y \mathbin{\#} -Y$.
\end{thm}

The problem of embedding 3-manifolds into $S^4$ is closely related to the doubly sliceness of a knot. A knot is called \emph{smoothly slice} if it arises as the cross-section of a smooth 2-knot in $S^4$. Furthermore, it is called \emph{smoothly doubly slice} if it arises as the cross-section of a smoothly \emph{unknotted} 2-knot in $S^4$. Thus Zeeman's proof on 1-twist spinning of $K$, implies that $K \mathbin{\#} -K$ is smoothly doubly slice. Moreover, if a knot is smoothly doubly slice, then as observed earlier, its double-branched cover smoothly embeds in $S^4$. In particular, we have the following corollary:

\begin{thm}[{\cite{Donald:2015-1}}]\label{thm:donald} Let $\{K(p_i,q_i)\}_{1\leq i \leq n}$ be a collection of 2- bridge knots. If $K = \#_{i=1}^n K(p_i,q_i)$, then $K$ is smoothly doubly slice    if and only if there exists a knot $J$ such that $K$ is equivalent to $J \mathbin{\#} -J$.
\end{thm}

We comment on the proof of Theorem~\ref{thm:donald} and the relevance of Lisca's work \cite{Lisca:2007-1, Lisca:2007-2} in relation to the theorem. If a connected sum of lens spaces $L = \#_{i=1}^n L(p_i,q_i)$ smoothly embeds in $S^4$, then it separates $S^4$ into two smooth compact manifolds $X$ and $X'$. Moreover, it can be easily shown, using the Mayer-Vietoris sequence, that $X$ and $X'$ are rational homology balls; that is, $H_*(X; \mathbb{Q}) \cong H_*(X'; \mathbb{Q}) \cong H_*(B^4; \mathbb{Q})$. Let $\natural_i X(p_i,q_i)$ denote the boundary connected sum of canonical negative-definite plumbings so that $\partial \left(\natural_i X(p_i,q_i)\right) \cong L$. By gluing $\natural_i X(p_i,q_i)$ with $X$, we obtain a smooth closed negative-definite 4-manifold $$W := \natural_i X(p_i,q_i) \cup_{L} -X.$$ The Mayer-Vietoris sequence then implies $b_2(\natural_i X(p_i,q_i)) = b_2(W)$. Similarly, using the boundary connected sum of canonical negative-definite plumbings $\natural_i X(p_i,p_i-q_i)$, we obtain another smooth closed negative-definite 4-manifold $$W' := \natural_i X(p_i,p_i-q_i) \cup_{-L} X,$$ with $b_2(\natural_i X(p_i,p_i-q_i)) = b_2(W')$. For simplicity, we refer to the standard negative-definite lattice $(\Z^n,\langle-1\rangle^n)$ as $\langle-1\rangle^n$. Donaldson's diagonalization theorem~\cite{Donaldson:1987-1} implies that both $W$ and $W'$ have the standard intersection forms. The inclusions $\natural_i X(p_i,p_i) \hookrightarrow W$ and $\natural_i X(p_i,p_i-q_i) \hookrightarrow W'$ induce morphisms of integral lattices
$$Q_{\natural_i X(p_i,p_i)}\hookrightarrow \langle-1\rangle^n \qquad \text{ and } \qquad Q_{\natural_i X(p_i,p_i-q_i)}\hookrightarrow \langle-1\rangle^{n'}$$
where $n = b_2(\natural_i X(p_i,p_i))$ and $n' = b_2(\natural_i X(p_i,p_i-q_i))$. The surprising result of Lisca~\cite{Lisca:2007-1,Lisca:2007-2} proved that this condition is equivalent to the existence of a rational homology ball bounded by $L$. Moreover, it also provides a classification for smoothly slice connected sums of 2-bridge knots. So far, we have not used the fact that there are two rational homology balls, $X$ and $X'$, which glue together to give $S^4$. These assumptions impose further restrictions on the morphisms of integral lattices, termed a \emph{linear double subset} in \cite{Donald:2015-1} (see also Section~\ref{subsec:SeiferFibered}). Combined with the combinatorial framework established by Lisca, Donald utilizes this stronger obstruction to prove Theorem~\ref{thm:donald}.

\subsection{Embedding Seifert fibered spaces in $S^4$ }\label{subsec:SeiferFibered}
Mazur~\cite{Mazur:1961-1} showed that if a contractible 4-manifold admits a handle decomposition with a single 1-handle and a single 2-handle, then its double is diffeomorphic to the standard $S^4$. In particular, if a homology sphere is bounded by such a Mazur manifold, then it smoothly embeds in $S^4$. This construction has been used to embed many Brieskorn spheres in $S^4$. For example, the families found by Akbulut and Kirby~\cite{Akbulut-Kirby:1979-1} and Casson and Harer~\cite{Casson-Harer:1981-1} (see \cite{Fickle:1984, Oguz:2024-1} for more examples):


\begin{thm}[{\cite{Akbulut-Kirby:1979-1,Casson-Harer:1981-1}}]\label{Thm:Mazur} The following Brieskorn homology spheres smoothly embed in $S^4$:
\begin{enumerate}

    \item $\Sigma(p, ps \pm 1, ps \pm 2)$ for $p$ odd, and
    \item $\Sigma(p, ps - 1, ps +1)$ for $p$ even and $s$ odd.
\end{enumerate}
\end{thm}

On the other hand, if a homology sphere $Y$ smoothly embeds in $S^4$, it separates $S^4$ into two smooth compact manifolds, $X$ and $X'$, which are both homology balls; that is, $H_*(X; \mathbb{Z}) \cong H_*(X'; \mathbb{Z}) \cong H_*(B^4; \mathbb{Z})$. In particular, if a homology sphere does not bound a smooth homology ball, it does not smoothly embed into $S^4$. Many gauge theoretical arguments have been developed to obstruct homology spheres from bounding $B^4$ (see, e.g., \cite{Fintushel-Stern:1985-1, Furuta:1990-1}). However, in this survey, we focus on results that obstruct the smooth embedding of Seifert fibered spaces into $S^4$, derived from Donaldson’s diagonalization theorem~\cite{Donaldson:1987-1}. We remark that the obstruction crucially utilizes the fact that there is a smooth embedding in $S^4$, rather than just the existence of a smooth homology ball filling.

Let $Y$ be a Seifert fibered space with base orbifold $S^2$. When $\varepsilon(Y)=0$, Donald~\cite{Donald:2015-1} (see also \cite{Hillman:2009-1}) obtains the following:

\begin{thm}[{\cite[Theorem 1.3]{Donald:2015-1}}]\label{Thm:e=0case} Let $Y$ be a Seifert fibered space over $S^2$ and $\varepsilon(Y)=0$. If $Y$ smoothly embeds in $S^4$, then $Y$ can be written in the form
$$Y= S^2(0,(p_1,q_1),(-p_1,q_1),(p_2,q_2),(-p_2,q_2)\ldots, (p_k,q_k),(-p_k,q_k)).$$
\end{thm}

For the case when $\varepsilon(Y)<0$, Issa and McCoy~\cite{Issa-McCoy:2020-1} obtained the following result:

\begin{thm}[{\cite[Theorem 1.1]{Issa-McCoy:2020-1}}]\label{Thm:e<0case} Let $Y= S^2(e,(p_1,q_1),(p_2,q_2),\ldots, (p_n,q_n))$ be a Seifert fibered space over $S^2$ in standard form and $\varepsilon(Y)<0$. If $Y$ smoothly embeds in $S^4$, then $e \geq - \frac{k+1}{2}$. Moreover, if $e = - \frac{k+1}{2}$, then $Y$ smoothly embeds in $S^4$ if and only if $Y$ can be written in the form
$$Y= S^2\left(e,(-a,a-1),(-a,1),(-a,a-1),\ldots, (-a,1),(-a,a-1)\right).$$
\end{thm}

In fact, both results are valid for Seifert fibered spaces with any orientable base surface (see \cite[Theorem 1.3]{Donald:2015-1} and \cite[Theorem 1.1]{Issa-McCoy:2020-1} for the precise statements). Moreover, as in the lens space cases, these results have implications on the double sliceness of knots whose double branched covers are Seifert fibered spaces (see \cite[Theorem 1.11 and Proposition 1.12]{Issa-McCoy:2020-1} for the consequences and \cite[Theorem 1.1]{McCoy-McDonald:2024-1} for related results).


The above theorem stem from the following idea, combined with delicate combinatorial arguments. If a Seifert fibered space $Y$ smoothly embeds in $S^4$, then it separates $S^4$ into two smooth compact manifolds, $X$ and $X'$. Recall that if $Y$ is in standard form, then it bounds a plumbed smooth negative semidefinite 4-manifold $X_\Gamma$ prescribed by a plumbing graph $\Gamma$ associated with $Y$. By gluing $X_\Gamma$ with $X$, we obtain a smooth closed negative-definite 4-manifold $$W := X_\Gamma \cup_{Y} -X.$$ Similarly, by gluing $X_\Gamma$ with $X'$, we obtain another smooth closed negative-definite 4-manifold $$W' := X_\Gamma \cup_{Y} -X'.$$ Moreover, Donaldson's diagonalization theorem~\cite{Donaldson:1987-1} implies that both $W$ and $W'$ have the standard intersection forms. The inclusions $X_\Gamma \hookrightarrow W$ and $X_\Gamma \hookrightarrow W'$ induce morphisms of integral lattices
$$Q_{\Gamma}\hookrightarrow \langle-1\rangle^n \qquad \text{ and } \qquad Q_{\Gamma}\hookrightarrow \langle-1\rangle^n,$$ where $n = b_2(X_\Gamma)$. These embeddings can be represented by integral matrices, with their transposes denoted as $A$ and $A'$, respectively. The inclusions $Y  \hookrightarrow W$ and $Y \hookrightarrow W'$ induces maps on cohomology: 
$$H^2 (W; \Z) \to H^2 (Y ; \Z) \qquad \text{ and } \qquad H^2 (W'; \Z) \to H^2 (Y ; \Z),$$
which can be understood in terms of the matrices $A$ and $A'$
More precisely, we may identify $H^2 (Y ; \Z)$ with $\Z^n/\mathrm{im}(Q_\Gamma)$ in such a way that the images of the above maps correspond to $\mathrm{im}(A)/\mathrm{im}(Q_\Gamma)$ and $\mathrm{im}(A')/\mathrm{im}(Q_\Gamma)$~\cite[Theorem 3.6]{Donald:2015-1}. This, combined with the fact that the inclusion maps induce an isomorphism on the torsion parts of $H^2(X; \mathbb{Z}) \oplus H^2(X'; \mathbb{Z})$ and $H^2(Y; \mathbb{Z})$, allows us to conclude the following:

\begin{thm}[{\cite[Theorem 3.9]{Donald:2015-1}, \cite[Theorem 5.2]{Issa-McCoy:2020-1}}] Let $Y$ be the boundary of a smooth, negative definite, plumbed 4-manifold $X_\Gamma$ prescribed by a plumbing graph $\Gamma$. If $Y$ smoothly embeds in $S^4$, then there exist morphisms of integral lattices
    $$Q_{\Gamma}\hookrightarrow \langle-1\rangle^n \qquad \text{ and } \qquad Q_{\Gamma}\hookrightarrow \langle-1\rangle^{n},$$ where $n = b_2(X_\Gamma)$, such that the augmented matrix $(A \vert A')$ is surjective, where $A$ and $A'$ are the transposes of the integer matrices representing the morphisms of integral lattices.
\end{thm}

\noindent This theorem is a key ingredient for 
Theorem~\ref{Thm:e<0case}.



\section{Embeddings in definite manifolds}
\subsection{Embeddings sums of lens spaces in definite manifolds}
In this section, we consider the problem of the existence of smooth embeddings of a closed 3-manifold $Y$ into smooth negative-definite manifolds, with a focus on the case when $Y$ is a lens space. Beyond lens spaces, it appears that little known about this problem \cite{Oguz:2024-2}.

Let $L(p,q)$ be a lens space with coprime integers $p$ and $q$ with $p>q>0$. Using the continued fraction expansion of $p/q$, Edmonds and Livingston~\cite{Edmonds-Livingston:1996-1} proved the following:

\begin{thm}[{\cite[Proposition 2.5]{Edmonds-Livingston:1996-1}}]\label{Thm:embeddingtocp2} Every lens space $L(p,q)$ smoothly embeds in $\#_n \overline{\mathbb{CP}}^2$ for some positive integer $n$.
\end{thm}

Many interesting natural questions arise from this fact. We list a few here:

\begin{ques}\label{questions}
Assume that every embedding is separating.
\begin{enumerate}[label=(\arabic*), font=\upshape]
    \item\label{qu1} Given a lens space $L(p,q)$, what is the minimal integer $n$ such that $L(p,q)$ smoothly embeds in $\#_n \overline{\mathbb{CP}}^2$, or in a smooth closed negative-definite 4-manifold $W$ with $b_2(W) = n$?
    \item\label{qu2} Does there exist an integer $n$ so that every lens space smoothly embeds in $\#_n \overline{\mathbb{CP}}^2$, or in a smooth closed negative-definite 4-manifold $W$ with $b_2(W) = n$?
\end{enumerate}
\end{ques}


There are some partial answers provided by Edmonds and Livingston~\cite{Edmonds-Livingston:1996-1} (see \cite[Section 13]{Edmonds-Livingston:1996-1} for a nice summary of their results), where they obstructed some lens spaces from being smoothly embedded into $\#_n \overline{\mathbb{CP}}^2$ for small $n$. Further investigations were made by the authors in \cite{Aceto-McCoy-Park:2022-1}. First, we state their main theorem. Here, we present only a portion of it.

\begin{thm}[{\cite[Theorem 1.1]{Aceto-McCoy-Park:2022-1}}]\label{Thm:AMP2022}
If $L=\#_i L(p_i,q_i)$ is a connected sum of lens spaces and $\natural_i X(p_i,q_i)$ is the corresponding boundary connected sum of canonical negative-definite plumbings, then the following are equivalent:
\begin{enumerate}[label=(\roman*), font=\upshape]
\item\label{it:min_filling} every smooth negative-definite filling $W$ of $L$ satisfies $b_2(W) \geq b_2(\natural_i X(p_i,q_i))$;
\item\label{it:combinatorial} the canonical negative-definite linear plumbing graph associated to $L$ does not contain any of the following configurations as an induced subgraph:
\begin{multicols}{2}
\begin{enumerate}[label=(\alph*),font=\upshape]
\item\label{it:4} $\begin{tikzpicture}[xscale=1.0,yscale=1,baseline={(0,0)}]
    \node at (1-0.1, .4) {$-4$};
    \node (A_1) at (1, 0) {$\bullet$};
  \end{tikzpicture}$
\item\label{it:52} $\begin{tikzpicture}[xscale=1.0,yscale=1,baseline={(0,0)}]
    \node at (1-0.1, .4) {$-5$};
    \node at (2-0.1, .4) {$-2$};
    \node (A_1) at (1, 0) {$\bullet$};
    \node (A_2) at (2, 0) {$\bullet$};
        \path (A_1) edge [-] node [auto] {$\scriptstyle{}$} (A_2);
  \end{tikzpicture}$
\item\label{it:622} $\begin{tikzpicture}[xscale=1.0,yscale=1,baseline={(0,0)}]
    \node at (1-0.1, .4) {$-6$};
    \node at (2-0.1, .4) {$-2$};
    \node at (3-0.1, .4) {$-2$};
    \node (A_1) at (1, 0) {$\bullet$};
    \node (A_2) at (2, 0) {$\bullet$};
    \node (A_3) at (3, 0) {$\bullet$};
        \path (A_1) edge [-] node [auto] {$\scriptstyle{}$} (A_2);
    \path (A_2) edge [-] node [auto] {$\scriptstyle{}$} (A_3);
  \end{tikzpicture}$
\item\label{it:2-2} $\begin{tikzpicture}[xscale=1.0,yscale=1,baseline={(0,0)}]
    \node at (1-0.1, .4) {$-2$};
    \node at (2-0.1, .4) {$-2$};
    \node (A_1) at (1, 0) {$\bullet$};
    \node (A_2) at (2, 0) {$\bullet$};
  \end{tikzpicture}$
\item\label{it:3-22}  $\begin{tikzpicture}[xscale=1.0,yscale=1,baseline={(0,0)}]
    \node at (1-0.1, .4) {$-3$};
    \node at (2-0.1, .4) {$-2$};
    \node at (3-0.1, .4) {$-2$};
    \node (A_1) at (1, 0) {$\bullet$};
    \node (A_2) at (2, 0) {$\bullet$};
    \node (A_3) at (3, 0) {$\bullet$};
    \path (A_2) edge [-] node [auto] {$\scriptstyle{}$} (A_3);
  \end{tikzpicture}$
\item\label{it:33}
$\begin{tikzpicture}[xscale=1.0,yscale=1,baseline={(0,0)}]
    \node at (1-0.1, .4) {$-3$};
    \node at (2-0.1, .4) {$-3$};
    \node (A1_1) at (1, 0) {$\bullet$};
    \node (A1_2) at (2, 0) {$\bullet$};
    \path (A1_1) edge [-] node [auto] {$\scriptstyle{}$} (A1_2);
  \end{tikzpicture}$
\item\label{it:323} $\begin{tikzpicture}[xscale=1.0,yscale=1,baseline={(0,0)}]
    \node at (1-0.1, .4) {$-3$};
    \node at (2-0.1, .4) {$-2$};
    \node at (3-0.1, .4) {$-3$};
    \node (A1_1) at (1, 0) {$\bullet$};
    \node (A1_2) at (2, 0) {$\bullet$};
    \node (A1_3) at (3, 0) {$\bullet$};
    \path (A1_2) edge [-] node [auto] {$\scriptstyle{}$} (A1_3);
    \path (A1_1) edge [-] node [auto] {$\scriptstyle{}$} (A1_2);
  \end{tikzpicture}$
\item\label{it:3223}
$\begin{tikzpicture}[xscale=1.0,yscale=1,baseline={(0,0)}]
    \node at (1-0.1, .4) {$-3$};
    \node at (2-0.1, .4) {$-2$};
    \node at (3-0.1, .4) {$-2$};
    \node at (4-0.1, .4) {$-3$};
    \node (A1_1) at (1, 0) {$\bullet$};
    \node (A1_2) at (2, 0) {$\bullet$};
    \node (A1_3) at (3, 0) {$\bullet$};
    \node (A1_4) at (4, 0) {$\bullet$};
    \path (A1_2) edge [-] node [auto] {$\scriptstyle{}$} (A1_3);
    \path (A1_3) edge [-] node [auto] {$\scriptstyle{}$} (A1_4);
    \path (A1_1) edge [-] node [auto] {$\scriptstyle{}$} (A1_2);
  \end{tikzpicture}$
\item\label{it:3532} 
$\begin{tikzpicture}[xscale=1.0,yscale=1,baseline={(0,0)}]
    \node at (1-0.1, .4) {$-3$};
    \node at (2-0.1, .4) {$-5$};
    \node at (3-0.1, .4) {$-3$};
    \node at (4-0.1, .4) {$-2$};
    \node (A1_1) at (1, 0) {$\bullet$};
    \node (A1_2) at (2, 0) {$\bullet$};
    \node (A1_3) at (3, 0) {$\bullet$};
    \node (A1_4) at (4, 0) {$\bullet$};
    \path (A1_2) edge [-] node [auto] {$\scriptstyle{}$} (A1_3);
    \path (A1_3) edge [-] node [auto] {$\scriptstyle{}$} (A1_4);
    \path (A1_1) edge [-] node [auto] {$\scriptstyle{}$} (A1_2);
  \end{tikzpicture}$
\item\label{it:2235} $\begin{tikzpicture}[xscale=1.0,yscale=1,baseline={(0,0)}]
    \node at (1-0.1, .4) {$-2$};
    \node at (2-0.1, .4) {$-2$};
    \node at (3-0.1, .4) {$-3$};
    \node at (4-0.1, .4) {$-5$};
    \node (A1_1) at (1, 0) {$\bullet$};
    \node (A1_2) at (2, 0) {$\bullet$};
    \node (A1_3) at (3, 0) {$\bullet$};
    \node (A1_4) at (4, 0) {$\bullet$};
    \path (A1_2) edge [-] node [auto] {$\scriptstyle{}$} (A1_3);
    \path (A1_3) edge [-] node [auto] {$\scriptstyle{}$} (A1_4);
    \path (A1_1) edge [-] node [auto] {$\scriptstyle{}$} (A1_2);
  \end{tikzpicture}$;
\end{enumerate}
\end{multicols}
\end{enumerate}
\end{thm}

Lens spaces that satisfy the conclusion of the theorem are referred to as \emph{$b_2$-minimal lens spaces}. First, we consider an immediate application to Question~\ref{questions} (providing only a partial answer to \ref{qu1}). For a more general statement, see \cite[Theorem 1.5]{Aceto-McCoy-Park:2022-1}.

\begin{cor}\label{cor:embeddingminimallensspace}
    Let $L(p,q)$ be a $b_2$-minimal lens space. Suppose $L(p,q)$ smoothly embeds in a smooth, closed, negative-definite 4-manifold $W$ as a separating manifold, then 
    $$ b_2(X(p,q)) \leq b_2(W).$$ In particular, for any integer $n$, there exist lens spaces that do not smoothly embed as a separating submanifold into any smooth, closed, negative-definite 4-manifold $W$ with $b_2(W) = n$ .
\end{cor}
\begin{proof}
    Suppose such an embedding exists. Then $L(p,q)$ separates $W$ into two smooth negative-definite compact manifolds $X$ and $X'$, where $\partial X = L(p,q)$ and $\partial X' = -L(p,q)$. Moreover, using the Mayer-Vietoris sequence, we have that $b_2(X) + b_2(X') = b_2(W)$. Since $L(p,q)$ is $b_2$-minimal, the following inequalities hold:
    $$b_2(X(p,q)) \leq b_2(X) \leq b_2(W).$$
    The latter part of the theorem follows from the existence of a sequence of $b_2$-minimal lens spaces ${L(p_i, q_i)}_{i \in \mathbb{N}}$ with arbitrarily large $b_2(X(p_i,q_i))$. This can be achieved by selecting $L(p_i, q_i)$ as a lens space where the continued fraction expansion of $p_i/q_i$ consists of integers greater than 4, with length $i$.
\end{proof}

We briefly outline the proof of Theorem~\ref{Thm:AMP2022}. The proof that \ref{it:min_filling} implies \ref{it:combinatorial} is the easier direction. Let us consider its contrapositive. If the plumbing graph associated with $\natural_i X(p_i,q_i)$ contains any of the induced subgraphs listed in \ref{it:combinatorial}, then the spherical generators of $\natural_i X(p_i,q_i)$ can be used to construct a smoothly embedded copy of $X(4,1)$, $X(9,2)$, $X(16,3)$, or $X(64,23)$ in $\natural_i X(p_i,q_i)$. Since the corresponding lens spaces $L(4,1)$, $L(9,2)$, $L(16,3)$, and $L(64,23)$ bound rational homology balls, we may remove a copy of $X(4,1)$, $X(9,2)$, $X(16,3)$, or $X(64,23)$ and replace it with a rational ball to produce a smaller definite filling.

For the converse, Donaldson's Diagonalization Theorem~\cite{Donaldson:1987-1} is again applied. Let $L = \#_i L(p_i, q_i)$ be a connected sum of lens spaces, and let $X$ be a smooth, negative-definite filling of $L$. Using the canonical negative-definite plumbing $\natural_i X(p_i, p_i-q_i)$, we obtain a smooth, closed, negative-definite 4-manifold
$$W := X \cup_{L} \natural_i X(p_i,p_i-q_i).$$
Donaldson's diagonalization Theorem~\cite{Donaldson:1987-1} implies that $W$ has the standard intersection form. The inclusions $X \hookrightarrow W$ and $\natural_i X(p_i, p_i-q_i) \hookrightarrow W$ induce a morphism of integral lattices
$$Q_{X} \oplus Q_{\natural_i X(p_i,p_i-q_i)}\hookrightarrow \langle-1\rangle^n$$
where $n = b_2(X)+b_2(\natural_i X(p_i,p_i-q_i))$. An easy argument using the Mayer-Vietoris exact sequence implies that $Q_X$ is isomorphic to the orthogonal complement of the image of $Q_{\natural_i X(p_i, p_i-q_i)}$ in $\langle-1\rangle^n$. The key part of the proof, which involves intricate combinatorics of integral lattices, is to demonstrate that if $L$ satisfies assumption \ref{it:combinatorial}, then an embedding of $Q_{\natural_i X(p_i, p_i-q_i)}$ into $\langle-1\rangle^n$ exists if and only if $n$ is at least $b_2(X(p_i, q_i)) + b_2(\natural_i X(p_i, p_i-q_i))$. This implies that $b_2(X) \geq b_2(X(p_i, q_i))$. In fact, a stronger result is proved in \cite[Theorem 1.1]{Aceto-McCoy-Park:2022-1}, which states that if $L$ satisfies assumption \ref{it:combinatorial} and there is an embedding of $Q_{\natural_i X(p_i, p_i-q_i)}$ into $\langle-1\rangle^n$, then the embedding is unique up to automorphisms.
 

\subsection{Embedding amphichiral lens spaces in  definite manifolds}
Here, we demonstrate that a specific class of lens spaces, namely amphichiral lens spaces, can all be embedded into a small, smooth, negative-definite 4-manifold—a property not shared by general lens spaces, as illustrated in Corollary~\ref{cor:embeddingminimallensspace}. To be more precise, we construct a smooth, negative-definite filling for each amphichiral lens space, with the first Betti number zero and the second Betti number two. This construction will enable us to prove Theorem~\ref{thm:amhichiralfilling} and Corollary~\ref{cor:amphichiralgenus}. We say that a 3-manifold $Y$ is \emph{amphichiral} if there exists an orientation preserving diffeomorphism between $Y$ and its orientation reversal, i.e.,\ $Y\cong -Y$. Recall that $-L(p,q)\cong L(p,p-q)$ and that $L(p,q)\cong L(p',q')$ if and only if $p=p'$ and $q'\equiv q \bmod{p}$ or $q'\equiv q^{*} \bmod{p}$ where $q^*$ is the unique integer such that $p>q^*>0$ and $qq^*\equiv 1 \bmod{p}$. Since we are assuming that $p$ and $q$ are relatively prime with $p>q>0$, it follows that $L(p,q)$ is amphichiral if and only if $q^2\equiv -1 \bmod{p}$.

Our first goal is to characterize amphichiral lens spaces in terms of their canonical plumbing graph. Given two strings of integers
$(a_1,\dots,a_n)$ and $(b_1,\dots,b_m)$ we consider the following operations
\begin{itemize}
\item $(a_1,\dots,a_n),(b_1,\dots,b_m)\mapsto (a_1+1, a_2, \dots,a_n),(b_1,\dots,b_m,2)$
\item $(a_1,\dots,a_n),(b_1,\dots,b_m)\mapsto (2,a_1,\dots,a_n),(b_1,\dots,b_{m-1},b_m+1)$
\end{itemize}
We say that the pairs $$(a_1+1, a_2, \dots,a_n),(b_1,\dots,b_m,2)\qquad \text{ and } \qquad (2,a_1,\dots,a_n),(b_1,\dots,b_{m-1},b_m+1)$$ are obtained from
$(a_1,\dots,a_n),(b_1,\dots,b_m)$ via \emph{$2$-final expansions} and we call the inverses of these operations \emph{$2$-final contractions}. We use the same terminology for a single string, e.g.\ we say that $(a_1+1, a_2, \dots,a_n,2)$ is obtained from $(a_1,\dots,a_n)$ via $2$-final expansion. A pair of strings of integers are called \emph{complementary} if they are obtained from the pair $(2),(2)$ by a sequence of $2$-final expansions. A string of integers $(a_1,\dots,a_n)$ is called \emph{self-complementary} if the pair $(a_1,\dots,a_n),(a_1,\dots,a_n)$ is complementary. Note that for this definition the order of two strings does not matter since if $(a_1,\dots,a_n),(b_1,\dots,b_m)$ is obtained from $(2),(2)$ via $2$-final expansions, then $(b_1,\dots,b_m),(a_1,\dots,a_n)$ can also be obtained from $(2),(2)$ via $2$-final expansions by reversing the order of 2-final expansions and choosing the other operation for each step. 


In the next lemma, we provide several criteria for two strings of integers to be complementary.

\begin{lem}\label{lem:complementarycriteria}
Suppose $$p/q=[a_1,\dots,a_n]^- \qquad\text{ and }\qquad r/s=[b_1,\dots,b_m]^-$$
with $a_i, b_i\geq 2$, then the following conditions are equivalent. 
\begin{enumerate}[font=\upshape]
\item\label{item:1complementary} $(a_1,\dots,a_n)$, $(b_1,\dots,b_m)$ are complementary.
\item\label{item:2complementary} $(a_n,\dots,a_1)$, $(b_m,\dots,b_1)$ are complementary.
\item\label{item:3complementary} 
$[a_1,\dots,a_n,1,b_1,\dots,b_m]^-=0$.
\item\label{item:4complementary} We have 
\begin{equation}\label{cfidentity}
\frac{q^*}{p}+\frac{s}{r}=1
\end{equation}
where $q^*$ is the unique integer such that $p>q^*>0$ and $qq^*\equiv 1 \bmod p$.
\item\label{item:5complementary} $r=p$ and $s=p-q^*$.
\end{enumerate}
\end{lem}



\begin{proof}
For \eqref{item:1complementary} implies \eqref{item:2complementary}, we use induction on the number of 2-final contractions needed to reduce a given pair to $(2),(2)$. If the given pair of string is $(2),(2)$ then the statement is obvious. Now suppose the statement is true for all pairs of string obtained with at most $N$ 2-final expansions, with $N>0$. Choose a pair of strings obtained from $(2),(2)$ via $N+1$ 2-final expansions. We may write this pair as 
$$
(a_1+1,a_2, \dots,a_n),(b_1,\dots,b_m,2)
$$
with the proof in the other case being identical. Since $(a_1,\dots,a_n)$ and $(b_1,\dots,b_m)$
is obtained from $(2),(2)$ via $N$ 2-final expansions, we see that, applying the inductive hypothesis, the pair
$$
(a_n,\dots a_2,a_1+1),(2,b_m,\dots,b_1)
$$ 
are complementary, which completes the proof.

Next we show that \eqref{item:2complementary} implies \eqref{item:3complementary}. Suppose the strings $(a_n,\dots,a_1)$ and $(b_m,\dots,b_1)$ are complementary. Then, in terms of Neumann's plumbing calculus~\cite{Neumann}, the linear plumbing graph corresponding to the string 
$$(a_1,\dots,a_n,1,b_1,\dots,b_m)$$ is obtained from the one corresponding to 
$(2,1,2)$ via a series of blow-ups of edges. In particular, the underlying plumbed 3-manifold does not change. Now note that the plumbing graph coming from $(2,1,2)$
can be blown down twice to a 0-weighted single vertex. 

Regarding \eqref{item:3complementary} implies \eqref{item:4complementary}, this is essentially the implication $(2)$ implies $(3)$ in~\cite[Proposition~2.15]{Aceto:2020-1}. Following the definition 
of the continued fraction of a rooted plumbing graph given in~\cite{Aceto:2020-1} one obtains~\eqref{cfidentity}. 

Next, we verify that \eqref{item:4complementary} implies \eqref{item:5complementary}. From \eqref{item:4complementary}, we have $$\frac{s}{r}=\frac{p-q^*}{p}.$$ Since we are assuming coprime integers $p$ and $q$ satisfy $p>q>0$ and coprime integers $r$ and $s$ satisfy $r>s>0$, the equality given in \eqref{item:5complementary} follows.

Finally, we show that \eqref{item:5complementary} implies \eqref{item:1complementary}. Let $[a]^{m}$ denote the tuple
\[
\underbrace{a, a, \dots,a, a}_{\text{$m$ times}},
\]
with the understanding that $[a]^{0}$ denotes the empty tuple. Then we may write 
$$
\frac{p}{q}=[a_1,\dots,a_n]^-=\left[[2]^{c_0},d_1,[2]^{c_1},\dots,d_k,[2]^{c_k}\right]^-.
$$
Recall that if $[a_1,\dots,a_n]^-=p/q$ then $[a_n,\dots,a_1]^-=p/q^*$. Moreover,
we have $(p-q)^*=p-q^*$. 
Now, since $$\frac{p}{p-q^*} = [b_1,\dots,b_m]^-$$ we have
$$
\frac{p}{p-q}=[b_m,\dots,b_1]^-.
$$

In particular, the string $(b_m,\dots,b_1)$ can be determined from $(a_1,\dots,a_n)$ via the Riemenschneider point rule \cite{Riemenschneider:1974-1}. More precisely, we have
\begin{equation*}\label{eq:Riemanschneider_long}
\frac{p}{p-q}=\begin{cases}
\left[c_0+1\right]^- &\text{if $k=0$}\\
\left[c_0+2, [2]^{d_1-3}, c_1+3, [2]^{d_2-3}, \dots,  [2]^{d_k-3}, c_k + 2\right]^- &\text{if $k\geq 1$}.
\end{cases}
\end{equation*}
Let us look at the case $k>1$ (we leave the case $k=0$ and $k=1$ as an easy exercise). We can write the pair $(b_1,\dots,b_m),(a_1,\dots,a_n)$ as 
$$
(c_k + 2, [2]^{d_k-3}, c_{k-1}+3,[2]^{d_{k-1}-3},\dots,[2]^{d_{1}-3},c_0+2  ), ([2]^{c_0},d_1,[2]^{c_1},\dots,[2]^{c_{k-1}},d_k,[2]^{c_k}).
$$
We can verify, by induction on $k$, that these strings are obtained via 2-final expansions from the pair $(2),(2)$. A sequence of $c_k$ 2-final contractions on the right results in
$$
(2, [2]^{d_k-3}, c_{k-1}+3,[2]^{d_{k-1}-3},\dots,[2]^{d_{1}-3},c_0+2), ([2]^{c_0},d_1, [2]^{c_1},\dots, [2]^{c_{k-1}},d_k).
$$ 
From here, after $d_k-2$ 2-final contractions on the left we obtain
$$
(c_{k-1}+3, [2]^{d_{k-1}-3},\dots, [2]^{d_{1}-3},c_0+2), ([2]^{c_0},d_1,[2]^{c_1},\dots, [2]^{c_{k-1}+1}).
$$ 
Lastly, a 2-final contractions on the right gives
$$
(c_{k-1}+2, [2]^{d_{k-1}-3},\dots, [2]^{d_{1}-3},c_0+2), ([2]^{c_0},d_1, [2]^{c_1},\dots, [2]^{c_{k-1}}).
$$ 
This last pair has the same form as our initial pair and so we may apply the induction hypothesis.
This completes the proof of the lemma.
\end{proof}

\begin{lem}\label{lem:selfcomplementary} If $L(p,q)$ is a lens space
with $p/q=[a_1,\dots,a_n]^-$, then $L(p,q)$ is amphichiral if and only if $(a_1,\dots,a_n)$ is self-complementary.
\end{lem}

\begin{proof}
If $(a_1,\dots,a_n)$ is self-complementary, then by \eqref{item:5complementary} in Lemma~\ref{lem:complementarycriteria}, we obtain $q=p-q^*$ and therefore we have
$$
-L(p,q)\cong L(p,p-q) \cong L(p,p-q^*)\cong L(p,q).
$$
Conversely, suppose $L(p,q)$ is amphichiral. Since the oriented diffeomorphism type of $L(p,q)$ determines the string $(a_1,\dots,a_n)$ up to an overall reversal of order, we see that either $$[a_1,\dots,a_n]^-=\frac{p}{p-q} \qquad\text{ or }\qquad [a_1,\dots,a_n]^-=\frac{p}{p-q^*}.$$ If the first equality holds, then $p=2$ and $q=1$. In particular, the corresponding string is $(2)$ which is self complementary. If the second equality holds, the conclusion follows from \eqref{item:5complementary} implies \eqref{item:1complementary} in Lemma \ref{lem:complementarycriteria}.\end{proof}

\begin{lem}\label{amphichiralstring}
A lens space $L(p,q)$ is amphichiral if and only if one of the following holds
\begin{itemize}
\item $L(p,q)\cong L(2,1)$
\item the string associated with $L(p,q)$ is either of the form 
$$
(a_1,\dots,a_n,b_1+1,b_2,\dots,b_m)
\qquad \text{ or }\qquad
(a_1,\dots,a_n+1,b_1,b_2,\dots,b_m)
$$ 
for some pair of complementary strings $(a_1,\dots,a_n)$, $(b_1,\dots,b_m)$.
\end{itemize} 
\end{lem}
\begin{proof}
It can be easily verified that $L(2,1)$ is the unique amphichiral lens space with length one associated string. Hence, from now on we assume that the string of $L(p,q)$ has length greater than one. 


Suppose the string associated with $L(p,q)$ is either 
$$
(a_1,\dots,a_n,b_1+1,b_2,\dots,b_m)
\qquad \text{ or }\qquad
(a_1,\dots,a_n+1,b_1,b_2,\dots,b_m),
$$ 
where $(a_1,\dots,a_n)$, $(b_1,\dots,b_m)$ are complementary. Note that $(a_1,\dots,a_n,b_1+1,b_2,\dots,b_m)$ is obtained from $(2,3)$ via 2-final expansions and, similarly, $(a_1,\dots,a_n+1,b_1,b_2,\dots,b_m)$ is obtained from $(3,2)$. Moreover, since $(2,3)$ and $(3,2)$ are self-complementary, it follows that both cases of strings are also self-complementary. Hence, we conclude that $L(p,q)$ is amphichiral by Lemma~\ref{lem:selfcomplementary}.



Suppose $L(p,q)$ is amphichiral, and let $(c_1,\dots,c_k)$ be the string associated to $L(p,q)$. By Lemma~\ref{lem:selfcomplementary}, we have that $(c_1,\dots,c_k)$ is self-complementary. Hence it is enough to show that for each self-complementary chain $(c_1,\dots,c_k)$ with $k\geq 2$, the chain is obtained from either $(2,3)$ or $(3,2)$ via $2$-final expansions. We prove this by induction on the length of the string $k$. For the base case, it is straightforward to check that $(c_1,c_2)$ being self-complementary implies that $(c_1,c_2)$ is either $(2,3)$ or $(3,2)$. Now, suppose the thesis holds for any length smaller that $k$. Since we are assuming that $(c_1,\dots,c_k)$ is self-complementary, we may perform a $2$-final contraction on $(c_1,\dots,c_k)$ to obtain either $$(c_2,\dots, c_{k-1},c_k-1) \qquad\text{ or }\qquad (c_1-1,c_2,\dots,c_{k-1}).$$ Suppose we are in the first case (the other one being analogous) and claim that it is again self-complementary. Note that in this case $c_1=2$ and we have that the strings
$$
(c_2,\dots,c_k),(c_1,c_2, \dots,c_{k-1},c_k-1)
$$
are complementary since it is obtained by a $2$-final contraction from $(c_1,\dots,c_k),(c_1, \dots,c_k)$. It follows from condition~\eqref{item:2complementary} in Lemma~\ref{lem:complementarycriteria} that the strings
$$
(c_k,\dots,c_2),(c_k-1,c_{k-1},\dots,c_2,c_1)
$$
are complementary. Performing a $2$-final contraction on this pair we obtain the pair of complementary strings
$$
(c_k-1,c_{k-1},\dots,c_2),(c_k-1,c_{k-1},\dots,c_2).
$$
Again, by condition~\eqref{item:2complementary} in Lemma \ref{lem:complementarycriteria}, we conclude that $(c_2,\dots,c_{k-1},c_k-1)$ is self-complementary, and the claim follows from the inductive hypothesis.\end{proof}

We are ready to prove Theorem~\ref{thm:amhichiralfilling}, whose statement we recall.

\begin{thm}\label{thm:bodyamhichiralfilling}
Every amphichiral lens space admits a smooth, negative-definite filling $W$ with $b_1(W) = 0$ and $b_2(W) = 2$. \end{thm}
\begin{proof}
Let $L(p,q)$ be an amphichiral lens space. If $L(p,q) \cong L(2,1)$, then it is clear that it admits a smooth, negative-definite filling $X$ with $b_1(X)=0$ and $b_2(X)=2$ (in fact it bounds $b_2(X)=1$). Therefore, we assume that $L(p,q) \ncong L(2,1)$.  We will construct a cobordism $W$ such that $\partial W=-L(2,1)\sqcup L(p,q)$, $b_1(W)=0$, and $b_2(W)=1$. Given such $W$ we obtain a smooth negative-definite filling $X$ by capping $W$ off one boundary component with $X(2,1)$,  the canonical negative-definite plumbing bounded by $L(2,1)$, if $W$ is negative-definite and with  $-X(2,1)$ if $W$ is positive-definite.

	
	Assume that the string associated with $L(p,q)$ is of the form $(a_1,\dots,a_n+1,b_1,b_2,\dots,b_m)$ where strings $(a_1,\dots,a_n)$, $(b_1,\dots,b_m)$ are complementary which is one of the possibility in Lemma~\ref{amphichiralstring}; it will be clear from the proof that the other case stated in Lemma~\ref{amphichiralstring} can be treated in the same way.
 
 We start by adding a 1-handle to $L(2,1)\times [0,1]$ as shown in Figure~\ref{fig:first}. Before adding 2-handles, we redraw the boundary (i.e., $L(2,1)\# S^1\times S^2$) as shown in Figure~\ref{fig:second}. The fact that this surgery description corresponds to $L(2,1)\# S^1\times S^2$ follows from condition \eqref{item:3complementary} in Lemma~\ref{lem:complementarycriteria}. Now we attach two 2-handles with framings $0$ and $+1$ respectively as shown in Figure~\ref{fig:third}. The result is our cobordism $W$. By performing a simple handle, we can identify the positive boundary of $W$ with $L(p,q)$, which is depicted in Figure~\ref{fig:fourth}.

	In order to conclude that $b_1(W)=0$ and $b_2(W)=1$, it is enough to verify that at least  one of the two unknots representing the 2-handles are linked nontrivially with the 1-handle. This follows from the fact that the positive boundary of $W$ is a rational homology sphere.
\end{proof}



\begin{figure}\label{fig:cobordism}
\centering
\begin{subfigure}[b]{\textwidth}
	\centering
    \includegraphics[width=0.7\textwidth]{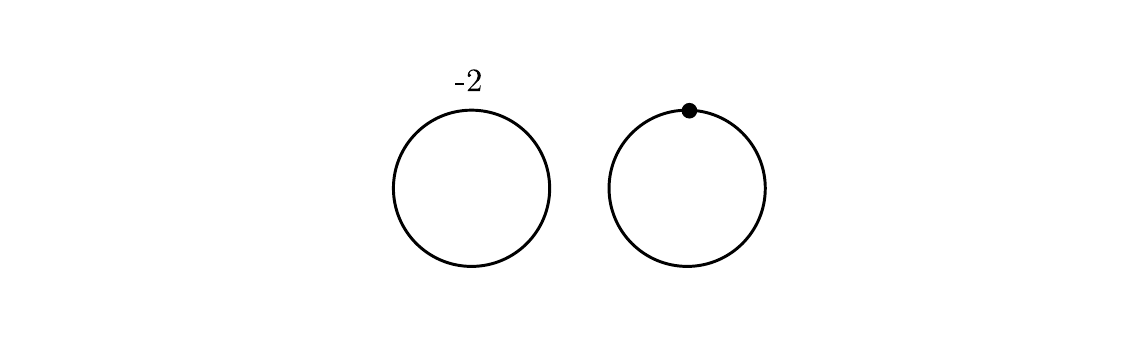}
    \caption{Attaching a 1-handle to $L(2,1)\times I$.}
    \label{fig:first}
\end{subfigure}\\
\vspace{0.5cm}
\begin{subfigure}{0.7\textwidth}
    \includegraphics[width=\textwidth]{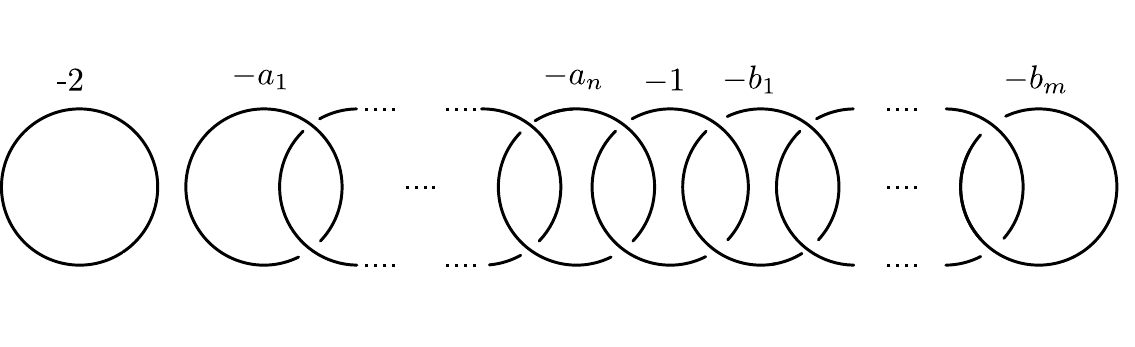}
    \caption{A different surgery diagram for $L(2,1)\# S^1\times S^2$.}
    \label{fig:second}
\end{subfigure}\\
\vspace{0.8cm}
\begin{subfigure}{0.7\textwidth}
    \includegraphics[width=\textwidth]{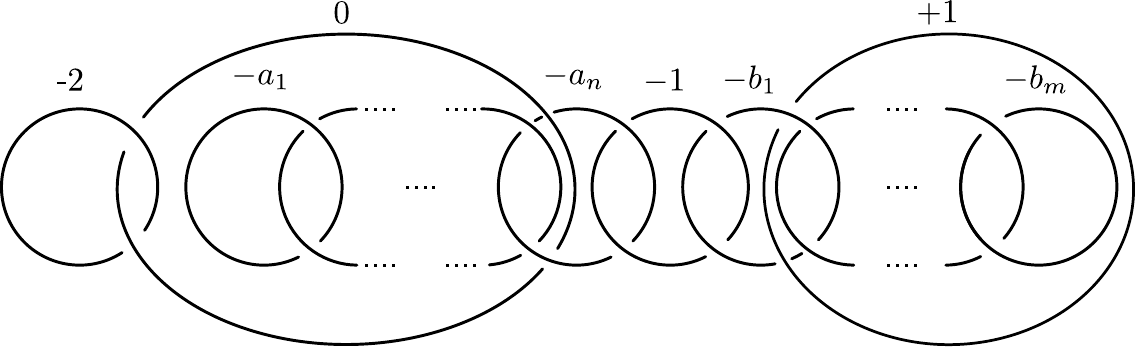}
    \vspace{0.2cm}
    \caption{Two 2-handle attachments.}
    \label{fig:third}
\end{subfigure}\\
\vspace{0.5cm}
\begin{subfigure}{\textwidth}
	\centering
    \includegraphics[width=0.7\textwidth]{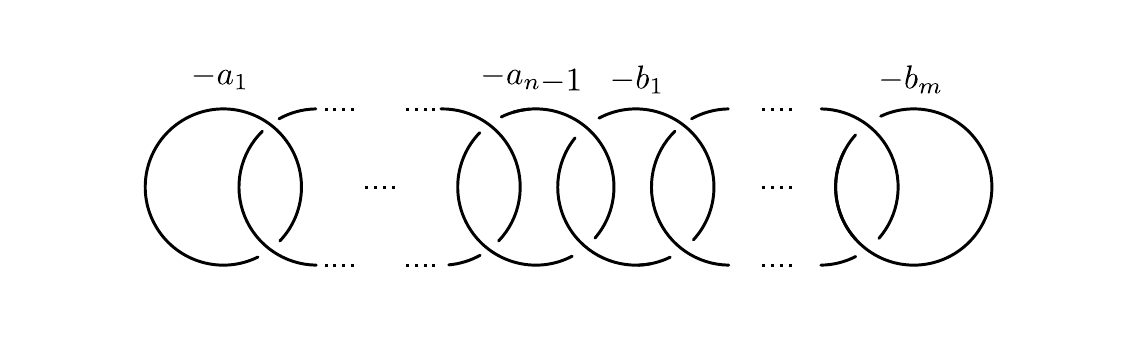}
    \caption{A different surgery diagram for $L(p,q)$.}
    \label{fig:fourth}
\end{subfigure}        
\caption{A cobordism between $L(2,1)$ and the amphichiral lens space $L(p,q)$ with associated string $(a_1,\dots,a_{n-1},a_{n}+1,b_1,\dots,b_{m})$.}
\label{fig:figures}
\end{figure}

We restate and prove Corollary~\ref{cor:amphichiralgenus}.
\begin{cor}\label{cor:amphichiralgenusbody}
The smooth 4-dimensional genus of an amphichiral 2-bridge knot is at most one.
\end{cor}	

\begin{proof}Let $L(p,q)$ be an amphichiral lens space and 
let $K(p,q)$ be the unique amphichiral 2-bridge knot such that its double-branched cover is  $L(p,q)$~\cite{Hodgson-Rubinstein:1985-1}. In the previous proof, we constructed a cobordism $W$ with $\partial W=-L(2,1)\sqcup L(p,q)$, $b_1(W)=0$, and $b_2(W)=1$. The cobordism can be obtained as the double cover of $S^3\times [0,1]$ branched along a properly embedded orientable surface $\Sigma$ such that 
\begin{itemize}
\item $\partial\Sigma=-K(2,1)\sqcup K(p,q)$ and
\item the projection on the time coordinate restricts to a Morse function for $\Sigma$ with one local minima, two saddle points and no local maxima.
\end{itemize} 
Here, the knot $-K(2,1)$ is the reverse of the mirror image of $K(2,1)$. This is possible since the attachment of the 2-handles was done equivariantly under the obvious involution shown in Figure~\ref{fig:third}. The double-branched cover of the local minima corresponds to the attachment of the 1-handle, and the double-branched cover of the two saddle points corresponds to the two 2-handle attachments (see, e.g., \cite[Chapter 6.3]{Gompf-Stipsicz:1999-1}). Note that we may choose the orientation of the Hopf link so that this cobordism is orientable.

With the orientation of the Hopf link that coincides with the orientation of $\Sigma$, we find $A \subset B^4$, a smoothly and properly embedded annulus with boundary the Hopf link. The gluing 
$$
(B^4,A)\cup_{(S^3\times\{0\},K(2,1))}(S^3\times [0,1],\Sigma)
$$   
realizes an orientable surface smoothly and properly embedded in $B^4$ with genus one and with boundary the knot $K(p,q)$.\end{proof}

\bibliographystyle{alpha}
\def\MR#1{}
\bibliography{bib}

\newpage
\end{document}